\newcommand{\open}{\Bbb}

\newcommand{\oN}{{\open N}}

\def\vdashd{\vdash_{\mathcal{D}}}

\def\qed{$\Box$\medskip
}

\def\dom{{\rm Dom}}

\def\D{\cal D}
\def\dep{=\!\!}

\documentclass[12pt]{article}
\usepackage{latexsym}
\usepackage{amsmath}
\usepackage{amssymb}
\usepackage{enumerate}
\usepackage{amsthm} 
\usepackage{graphicx} 
\usepackage{amssymb,latexsym,times}
\newtheorem{theorem}{Theorem}
\newtheorem{definition}[theorem]{Definition}
\newtheorem{lemma}[theorem]{Lemma}
\newtheorem{proposition}[theorem]{Proposition}
\newtheorem{corollary}[theorem]{Corollary}
\newtheorem{example}[theorem]{Example}

\def\={=\!\!}

\usepackage{proof}
\newcommand{\df}{\mathcal{D}}
\newcommand{\rel}{rel}

\newcommand{\mA}{{\mathfrak A}}

\newcommand{\FO}{{\rm FO}}
\newcommand{\Fr}{{\rm Fr}}
\newcommand{\Var}{{\rm Var}}
\newcommand{\len}{{\rm len}}

\title{Axiomatizing first order consequences in dependence logic}
\author{Juha Kontinen\thanks{Supported by grant 127661 of the Academy of Finland}\\  \and Jouko V\"a\"an\"anen\thanks{Research partially supported by
grant 40734 of the Academy of Finland and by the EUROCORES LogICCC LINT programme.}}
\begin{document}
\maketitle

\begin{abstract}
Dependence logic, introduced in \cite{MR2351449}, cannot be axiomatized. However, first-order consequences of dependence logic sentences can be axiomatized, and this is what we shall do in this paper. We give an explicit axiomatization and prove the respective Completeness Theorem. \end{abstract}

\section{Introduction}

Dependence logic was introduced in \cite{MR2351449}. It extends ordinary first order logic by new atomic formulas $\=(x_1,...,x_n,y)$ with the intuitive meaning that the values of the variables $x_1,...,x_n$  completely determine  the value of $y$ is. This means that the relevant semantic game is a game of imperfect information. A player who picks $y$ and claims that her strategy is a winning strategy should make the choice so that if the strategy is played twice, with the same values for $x_1,...,x_n$, then the value of $y$ is the same as well.  Dependence logic cannot be axiomatized, for the set of its valid formulas is of the same complexity as that of full second order logic. However, the first order consequences of dependence logic sentences can be axiomatized. In this paper we give such an axiomatization.

Let us quickly review the reason why dependence logic cannot be effectively axiomatized. Consider the sentence
\[\theta _1: \exists z \forall x \exists y(\dep(y,x)\wedge \neg y=z).\]
We give the necessary preliminaries about dependence logic in the next section, but let us for now accept that $\theta_1$ is true in a model if and only if the domain of the model is infinite. The player who picks $y$ has to pick a different $y$ for different $x$. Although dependence logic does not have a negation in the sense of classical logic, the mere existence of $\theta_1$ in dependence logic should give a hint that axiomatization is going to be a problem. Elaborating but a little, $\theta_1$ can be turned into a sentence $\theta_2$ in the language of arithmetic which says that some elementary axioms of number theory fail or else some number has infinitely many predecessors. We can now prove that a first-order sentence $\phi$ of the language of arithmetic is true in $(\mathbb{N},+,\times,<)$ if and only if $\theta_2\vee \phi$ is logically valid (true in every model) in dependence logic.  This can be seen as follows: Suppose first $\phi$ is true in $(\mathbb{N},+,\times,<)$. Let us take an arbitrary model $M$ of the language of arithmetic. If $M\models \theta_2$, we may conclude $M\models \Theta_2\vee \phi$.  So let us assume $M\not \models \theta_2$. Thus $M$ satisfies the chosen elementary axioms of number theory and every element has only finitely many predecessors. As a consequence, $M\cong (\mathbb{N},+,\times,<)$, so $M\models \phi$, and again $M\models \theta_2\vee \phi$.
For the converse, suppose  $\theta_2\vee \phi$ is logically valid. Since 
$(\mathbb{N},+,\times,<)$ fails to satisfy $\theta_2$, we must conclude that $\phi$ is true in $(\mathbb{N},+,\times,<)$. 

The above inference demonstrates that truth in $(\mathbb{N},+,\times,<)$ can be reduced to logical validity in dependence logic. Thus, by Tarski's Undefinability of Truth argument, logical validity in dependence logic is non-arithmetical, and there cannot be any (effective) complete axiomatization of dependence logic.

The negative result just discussed would seem to frustrate any attempt to axiomatize dependence logic. However, there are at least two possible remedies. The first is to modify the semantics - this in the line adopted in Henkin's Completeness Theorem for second-order logic. For dependence logic this direction is taken in Galliani \cite{btxdoc}. The other remedy is to restrict to a fragment of dependence logic. This is the line of attack of this paper. We restrict to logical consequences $T\models \phi$, in which $T$ is in dependence logic but $\phi$ is in first-order logic. 

The advantage of restricting to $T\models \phi$, with first-order $\phi$, is that we can reduce the Completeness Theorem,  assuming that $T\cup \{\neg \phi\}$ is deductively consistent, to the problem of constructing a model for $T\cup \{\neg \phi\}$. Since dependence logic can be translated to existential second-order logic, the construction of a model for $T\cup \{\neg \phi\}$ can in principle be done in first-order logic, by translating $T$ to first-order by using new predicate symbols. This observation already shows that $T\models \phi$, for first-order $\phi$, can {\em in principle} be axiomatized. Our goal in this paper is to give an explicit axiomatization.  

The importance of an {\em explicit} axiomatization over and above the mere knowledge that an axiomatization exists, is paramount. The axioms and rules that we introduce throw light in a concrete way on logically sound inferences concerning dependence concepts. It turns out, perhaps unexpectedly, that fairly simple albeit non-trivial axioms and rules suffice.

Our axioms and rules are based on Barwise \cite{MR0465788}, where approximations of Henkin sentences, sentences which start with a partially ordered quantifier, are introduced. The useful method introduced by Barwise
builds on earlier work on game expressions by Svenonius  \cite{MR0209138}
and Vaught \cite{MR0409106}. 

By axiomatizing first order consequences we get an axiomatization of inconsistent dependence logic theories as a bonus, contradiction being itself expressible in first order logic. The possibility of axiomatizing inconsistency in IF logic---a relative of dependence logic---has been emphasized by Hintikka~\cite{MR1410063}. 

The structure of the paper is the following. After the preliminaries we present our system of natural deduction in Section 3. In Section 4 we give  a rather detailed proof of the Soundness of our system, which is not a priori obvious. Section 5 is devoted to the proof, using game expressions and their approximations, of the Completeness Theorem. The final section gives examples and open problems.

The second author is indebted to John Burgess for suggesting the possible relevance for dependence logic of the work of Barwise on approximations of Henkin formulas.

\section{Preliminaries}

In this section we define Dependence Logic ($\df$) and recall some basic results about it. 

\begin{definition}[\cite{MR2351449}]
The syntax of $\df$ extends the syntax of $\FO$, defined in terms
of $\vee$, $\wedge$, $\neg$, $\exists$ and $\forall$, by new
atomic  formulas (dependence atoms) of the form
\begin{equation}\label{dep}\dep(t_1,\ldots,t_n),
\end{equation} where $t_1,\ldots,t_n$ are terms.  For a vocabulary $\tau$, $\df[\tau]$  denotes the set of $\tau$-formulas of $\df$.
\end{definition}
The intuitive meaning of the dependence atom (\ref{dep}) is
that the value of the term $t_n$ is  functionally determined by the values of
the terms $t_1,\ldots, t_{n-1}$. As singular cases we have $\dep()$
 which we take to be universally true, and
 $\dep(t)$ meaning that the value of $t$ is constant. 

The set $\Fr(\phi)$ of free variables of a formula 
$\phi\in \df$ is defined  as for first-order logic, except that we have the new case
\[ \Fr(\dep(t_1,\ldots,t_n))=\Var(t_1)\cup\cdots \cup \Var(t_n), \] 
where $\Var(t_i)$ is the set of variables occurring in the term $t_i$. If $\Fr(\phi)=\emptyset$, 
we call $\phi$ a sentence.

In order to define the semantics of $\df$, we first need to
define the concept of a \emph{team}. Let $\mA$ be a model with domain $A$. 
 {\em Assignments} of $\mA$
are finite mappings from variables into $A$. The value of a term $t$
in an assignment $s$ is denoted by $t^{\mA}\langle s\rangle$.
If $s$ is an assignment, $x$ a variable,  and $a\in A$, then $s(a/x)$ denotes the 
assignment (with domain $\dom(s)\cup \{x\}$)  which agrees with $s$
everywhere except that it maps $x$ to $a$. 

Let $A$ be a set and $\{x_1,\ldots,x_k\}$ a finite (possibly empty) set  of
variables. 
 A {\em team} $X$ of $A$ with domain
$\dom(X)=\{x_1,\ldots,x_k\}$ is any set of assignments from the variables
$\{x_1,\ldots,x_k\}$ into the set $A$. We denote by $\rel(X)$ the
$k$-ary relation of $A$ corresponding to $X$ 
\[\rel(X)=\{(s(x_1),\ldots,s(x_k)) : s\in X \}.  \]
 If $X$ is a team of $A$,
and $F\colon X\rightarrow A$, we use $X(F/x_n)$ to denote the (supplemented) team
$\{s(F(s)/x_n) : s\in X \}$ and $X(A/x_n)$ the (duplicated) team $\{s (a/x_n)
: s\in X\ \textrm{and}\ a\in A \}$. It is convenient to adopt a shorthand notation for teams arising from successive applications of the supplementation and duplication operations, e.g., we abbreviate  $X(F_1/x_1)(A/x_2)(F_3/y_1)$ as 
$X(F_1AF_3/x_1x_2 y_1)$. 

Our treatment of negation is the following:
 We call a formula of $\df$ {\em first-order} if it does not contain any dependence atoms.   We assume that the scope of negation is always a first order formula. We could allow negation everywhere, but since negation in dependence logic is treated as dual, it would only result in the introduction of a couple of more rules of the de Morgan type in the definition of semantics, as well as in the definition of the deductive system. 

We are now ready to define the semantics of dependence logic. In this definition 
 $\mA\models_s\phi$  refers to  satisfaction in first-order logic. 
 
\begin{definition}[\cite{MR2351449}]\label{sat} Let
$\mA$ be a model and $X$ a team of $A$. The satisfaction relation
$\mA\models _X \varphi$ is defined as follows:
\begin{enumerate}
\item If $\phi$ is  first-order, then $\mA\models _X\phi$ iff for all $s\in X$,  $\mA\models_s\phi$. 

\item $\mA\models _X \dep(t_{1},\ldots,t_{n})$ iff for all $s,s'\in
X$ such that\\ $t_1^{\mA}\langle s\rangle  =t_1^{\mA}\langle
s'\rangle  ,\ldots, t_{n-1}^{\mA}\langle s\rangle
=t_{n-1}^{\mA}\langle s'\rangle  $, we have $t_n^{\mA}\langle
s\rangle  =t_n^{\mA}\langle s'\rangle  $.


\item $\mA\models _X \psi \wedge \phi$ iff $\mA\models _X \psi$ and $\mA\models _X \phi$.
\item $\mA\models _X \psi \vee \phi$ iff $X=Y\cup Z$ such that
 $\mA\models _Y \psi$  and $\mA\models _Z \phi$ .

\item   $\mA \models _X \exists x_n\psi$ iff $\mA \models _{X(F/x_n)} \psi$ for some $F\colon X\to A$.

\item $\mA \models _X \forall x_n\psi$ iff $\mA \models _{X(A/x_n)} \psi$.

\end{enumerate}
Above, we assume that the domain of $X$ contains the variables free in $\phi$. Finally, a sentence $\phi$ is true in a model $\mA$, $\mA\models \phi$,  if $\mA\models _{\{\emptyset\}} \phi$. 
\end{definition}

The truth definition of dependence logic can be also formulated in game theoretic terms \cite{MR2351449}. In terms of semantic games, the truth of $\=(x_1,\ldots,x_n,y)$ means that the player who claims a winning strategy has to demonstrate certain {\em uniformity}. This means that if the game is played twice, the player, say $\exists$, reaching both times the same subformula $\=(x_1,\ldots,x_n,y)$, then if the values of $x_1,\ldots,x_n$ were the same in both plays, the value of $y$ has to be the same, too. 

 Next we define the notions of logical consequence and  equivalence for formulas of dependence logic. 
\begin{definition}\label{equiv} Let $T$ be a set of formulas of dependence logic with only finitely many free variables. The formula  $\psi$ is a \emph{logical consequence} of  $T$ ,
\[T\models\psi,  \]
 if for all models $\mA$ and teams $X$, with $\Fr(\psi)\cup\bigcup_{\phi\in T} \Fr(\phi)\subseteq \dom(X)$, and $\mA\models_X T$ we have $\mA\models_X\psi$. The formulas  $\phi$ and $\psi$ are \emph{logically equivalent},
\[\phi\equiv \psi,   \]
 if $\phi\models \psi$ and $\psi\models \phi $. 
\end{definition}


The following basic properties of dependence logic will be extensively used in this article.

Let $X$ be a team with domain $\{x_1,\ldots,x_k\}$ and $V\subseteq \{x_1,\ldots,x_k\}$. Denote by
$X\upharpoonright V$ the team $\{s\upharpoonright V : s\in X\}$ with domain $V$. The following lemma 
 shows that
the truth of a formula  depends only on the interpretations of the variables occurring free in the formula.

\begin{proposition}\label{freevar}
Suppose $V\supseteq \Fr(\phi)$. Then $\mA \models _X\phi$ if and only if $\mA \models _{X\upharpoonright V} \phi$.
\end{proposition}

The following fact is also a very basic property of all formulas of dependence logic:

\begin{proposition}[Downward closure]\label{Downward closure}Let $\phi$ be a formula of dependence logic, $\mA$  a model, and $Y\subseteq X$ teams. Then $\mA\models_X \phi$ implies $\mA\models_Y\phi$. 
\end{proposition}



\section{A system of natural deduction}

We will next present  inference rules  that allow us to derive all first-order consequences of sentences of  dependence logic.

Here is the first set of rules that we will use.  The  substitution  of a term $t$ to the free occurrences of $x$ in $\psi(x)$ is denoted by $\psi(t/x)$. Analogously to first-order logic,  no variable of $t$ can become bound in such substitution.  

We use an abbreviation $\vec{x}=\vec{y}$ for the formula  $\bigwedge_{1\le i \le \len(\vec{x})} x_i=y_i$,
assuming of course  that  $\vec{x}$ and $\vec{y}$ are tuples  of the same length $\len(\vec{x})$. Furthermore, for an assignment $s$, and  a tuple of variables $\vec{x}=(x_1,\ldots, x_n)$, we  sometimes denote the tuple $(s(x_1),\ldots, s(x_n))$ by $s(\vec{x})$.

{\scriptsize
\begin{figure}
\begin{tabular}{|c|c|c|}\hline
&& \\
{\small
Operation} & {\small Introduction} & {\small Elimination
} \\
&& \\
 \hline
&&\\
Conjunction
&
$
\infer[{\mbox{\tiny$\wedge$ I}}]{A \wedge B}{A & B}$
&
$\infer[{\mbox{\tiny$\wedge$ E}}]A{A\wedge B}
\qquad
\infer[{\mbox{\tiny$\wedge$ E}}]B{A\wedge B}
$\\&&\\
\hline
&&\\
Dis\-junc\-tion
&
$
\infer[{\mbox{\tiny$\vee$ I}}]{A \vee B}{A}
\qquad
\infer[{\mbox{\tiny$\vee$ I}}]{A \vee B}{B}$
&
$\infer[{\mbox{\tiny$\vee$ E}}]{C}{
       A\vee B
       &
                                   \infer*{C}{
               [A]
       }
       &
      \infer*{C}{
               [B]
       }
       }
$

\\&&\\
&
&Condition 1.\\
&&\\
\hline

&&\\
Ne\-ga\-tion 
 
&
$
\infer[{\mbox{\tiny$\neg$ I}}]{\neg A}{
      \infer*{B\wedge\neg B}{
               [A]
       }
       }
$
&

$
\infer[{\mbox{\tiny$\neg$ E}}]{A}{\neg\neg A}
$
\\&&\\

&Condition 2.
&Condition 2. 
\\
&&\\
\hline

&&\\
Uni\-ver\-sal quantifier
&
$
\infer[{\mbox{\tiny$\forall$ I}}]{\forall x_iA}{
      A
       }

$
&
$
\infer[{\mbox{\tiny$\forall$ E}}]{A(t/x_i)}{\forall x_i A}
$
\\&&\\
&Condition 3.
&
\\
&&\\
\hline
&&\\
Existential quantifier
&
$
\infer[{\mbox{\tiny$\exists$ I}}]{\exists x_iA}{
      A(t/x_i)
       }

$
&
$
\infer[{\mbox{\tiny$\exists$ E}}]{B}{
       \exists x_iA
       &
                                   \infer*{B}{
               [A]
       }
       }$
\\&&\\
&
&Condition 4.\\
&&\\
\hline
\multicolumn{3}{|l|}{  }\\
\multicolumn{3}{|l|}{\ Condition 1. $C$ is first-order.
}\\
\multicolumn{3}{|l|}{\ Condition 2. The formulas are first-order.
}\\
\multicolumn{3}{|l|}{\ Condition 3. The variable \(x_i\) cannot appear free  in any non-discharged assumption
}\\
\multicolumn{3}{|l|}{\ used in the derivation of \(A\).}\\
 \multicolumn{3}{|l|}{\ Condition 4. The variable
 \(x_i\) cannot appear free  in \(B\) and in any non-discharged  
}\\
\multicolumn{3}{|l|}{\ assumption used in the derivation of  \(B\),
except in \(A\).}\\
\multicolumn{3}{|l|}{  }\\
\hline
\end{tabular}
\caption{The first set of rules.\label{ded}}
\end{figure}}

In addition to the rules of Figure~\ref{ded}, we also adopt the following rules:
\begin{definition}

\begin{enumerate}

 
\item\label{rule1} Disjunction substitution: 

\[
\infer[]{A\vee C}{
        A\vee B
       &
                                   \infer*{C}{
               [B]
       }
       }\]

\item\label{rule2}  Commutation and associativity of disjunction: 
\[\infer{A\vee B}{B\vee A}\hspace{19mm}\infer{A\vee(B\vee C)}{(A\vee B)\vee C}\]

\item\label{rule3} Extending scope: 
\[\infer{ \forall x (A \vee B)}{\forall x A \vee B}   \]
 where the  prerequisite for applying this rule is that $x$ does not appear free in $B$.

\item\label{rule4}  Extending scope:
\[\infer{ \exists x (A \vee B)}{\exists x A \vee B}   \]
 where the  prerequisite for applying this rule is that $x$ does not appear free in $B$.

\item\label{rule5} Unnesting:
\[\infer{\exists z( \dep(t_1,...,z,...,t_n)\wedge z=t_i)}{\dep(t_1,...,t_n)}   \]
 where $z$ is a new variable.

\item\label{rule6}  Dependence distribution: let  
\begin{eqnarray*}
A &=& \exists y_1\ldots \exists y_n(\bigwedge_{1\le j\le n}\dep(\vec{z}^j,y_j)\wedge C),\\
B &=& \exists y_{n+1}\ldots \exists y_{n+m}(\bigwedge_{n+1\le j\le n+m}\dep(\vec{z}^j,y_j)\wedge D).
\end{eqnarray*}
where $C$ and $D$ are quantifier-free formulas without dependence atoms, and $y_i$, for $1\le i \le n$, does not appear in $B$ and $y_i$, for $n+1\le i \le n+m$, does not appear in $A$.
Then,
\[\infer{\exists y_1 \ldots \exists y_{n+m}(\bigwedge_{1\le j\le n+m}\dep(\vec{z}^j,y_j)\wedge (C\vee D))}{A\vee B}   \]Note that the logical form of this rule is:
\[\infer{\exists \vec{y}\exists \vec{y'}(\bigwedge_{1\le j\le n+m}\dep(\vec{z}^j,y_j)\wedge (C\vee D))}{\exists  \vec{y}(\bigwedge_{1\le j\le n}\dep(\vec{z}^j,y_j)\wedge C)\vee \exists \vec{y'}(\bigwedge_{n+1\le j\le n+m}\dep(\vec{z}^j,y_j)\wedge D)}   \]

\item\label{rule7} 
Dependence introduction:
\[\infer{\forall y\exists x (\dep(\vec{z},x)\wedge A)}{\exists x \forall y A}   \]
 where $\vec{z}$ lists the variables in $\Fr(A)- \{x,y\}$.

\item\label{rule8}  Dependence elimination:
\[ \infer{
\begin{array}{l}
\forall\vec{x_0}\exists\vec{y_0}(B(\vec{x_0},\vec{y_0})\wedge\\
\forall\vec{x_1}\exists\vec{y_1}(B(\vec{x_1},\vec{y_1})\wedge
\bigwedge_{=\!(\vec{w}_0^p,y_{0,p})\in S}(\vec{w}_0 ^p=\vec{w}_1 ^p\to y_{0,p}=y_{1,p})))
\end{array} 
 }{\forall\vec{x_0}\exists\vec{y_0} (\bigwedge_{1\le j\le k}\dep(\vec{w}^{i_j},y_{0,{i_j}})\wedge B(\vec{x_0},\vec{y_0})),
}\]
where $\vec{x_l}=(x_{l,1},\ldots,x_{l,m})$ and $\vec{y_l}=(y_{l,1},\ldots,y_{l,n})$ for $l \in \{0,1\}$ ($\vec{w}_0^p$ and $\vec{w}_1^p$ are related analogously),  and
 the variables in $\vec{w}^{i_j}$ are contained in the set
\[  \{x_{0,1},\ldots,x_{0,m},y_{0,1},\ldots,y_{0,i_j-1} \}  .\]
 Furthermore, the set $S$  contains the conjuncts of 
\[\bigwedge_{1\le j\le k}\dep(\vec{w}^{i_j},y_{0,{i_j}}), \]
and  the dependence atom $\dep(x_{0,1},\ldots,x_{0,m},y_{0,p})$ for each of the variables $y_{0,p}$ ($1\le p \le n$) such that $y_{0,p} \notin \{y_{0,i_1},\ldots, y_{0,i_k}\}$. 

\item The usual identity axioms.




\end{enumerate}
\end{definition}
It is worth noting that  the elimination rule for disjunction is not correct in the context of dependence logic. Therefore we have to assume the  rules 1-4 regarding disjunction, which  are easily derivable in first-order logic.   Note also that the analogues of the rules 1-4 for conjunction need not be assumed since they are easily derivable from the other rules.   

Note that we do not assume the so called Armstrong's Axioms for dependence atoms. If we assumed them, we might be able to simplify the dependence elimination and the dependence distribution rules, but we have not pursued this line of thinking.

\section{The Soundness Theorem}

In this section we show that the  inference rules defined in the previous section are sound for dependence logic. 

\begin{proposition} Let $T\cup\{\psi\}$ be a set of formulas of dependence logic.  If $T\vdashd \psi$, then $T \models \psi$.
\end{proposition}

\begin{proof} We will prove the claim using  induction on the length of derivation. The soundness of the rules $\neg$ E, and 2-6
follows from the corresponding logical equivalences proved in \cite{MR2351449} and \cite{ADJK} (rules 5-6).  
Furthermore, the soundness of the rules  $\wedge$ E, $\wedge$ I, $\vee$ I, and rule 1 is obvious. We consider the remaining rules below. The following
lemma is needed in the proof.


\begin{lemma}\label{terms} Let $\phi(x)$ be a formula, and $t$ a term such that in the substitution $\phi(t/x)$ no variable of $t$ becomes bound. Then for all $\mA$ and teams $X$, where $(\Fr(\phi)-\{x\})\cup \Var(t)\subseteq \dom(X)$\[  \mA\models _X \phi(t/x) \Leftrightarrow   \mA\models _{X(F/x)} \phi(x),  \]
where $F\colon X\rightarrow A$ is defined by  $F(s)= t^{\mA}\langle s\rangle$.
\end{lemma}
\begin{proof} Analogous to Lemma 3.28 in  \cite{MR2351449}. 
\end{proof}


\begin{itemize}
\item[$\vee$ E]    Assume that we have a natural deduction proof of a first-order formula $C$ from the assumptions 
\begin{equation*}
 \{A_1,\ldots,A_k\}
\end{equation*}
with the last rule $\vee$ E applied to  $A\vee B$.  Let $\mA$ and $X$ be such that  $\mA\models _X A_i$, for $1\le i\le k$. By the assumption, we have a shorter deduction of  $A\vee B$ from the same assumptions, and deductions
of $C$ from both  of the sets $\{A,A_1,\ldots,A_k\}$ and $\{B,A_1,\ldots,A_k\}$. By the induction assumption, we get that  $\mA\models _X A\vee B$, and hence $X=Y\cup Z$ with $\mA\models _Y A$ and $\mA\models _Z B$. Let $s\in X$, e.g. $s\in Y$. We know   $\mA\models _Y A$. Thus by the induction assumption, we get that  
$\mA\models _Y C$, and therefore $\mA\models_s C$. Analogously, if $s\in Z$, then since  $\mA\models _Z B$ we get $\mA\models _Z C$, and therefore $\mA\models_s C$. In either case $\mA\models _s C$, hence  $\mA\models _X C$ as wanted.

\item[$\neg$ I]  Assume that we have a natural deduction proof of a first order formula $\neg A$ from the  assumptions 
\begin{equation*}
 \{A_1,\ldots,A_k\}
\end{equation*}
with the last rule $\neg $ I.  Let $\mA$ and $X$ be such that  $\mA\models _X A_i$, for $1\le i\le k$. By the assumption, we have a shorter deduction of  $B\wedge\neg B$ from the assumptions $ \{A,A_1,\ldots,A_k\}$. We claim that now $\mA\models _X \neg A$, i.e., $\mA\models _s \neg A$ for all $s\in X$. For contradiction, assume that  $\mA\not\models _s \neg A$ for some $s\in X$. Then $\mA\models _s \ A$. By Proposition \ref{Downward closure}, we get that   $\mA\models _{\{s\}} A_i$, for $1\le i\le k$. Now, by the induction assumption, we get that  $\mA\models _{s} B\wedge\neg B$ which is a contradiction.

\item[ $\exists$ E]  Assume that we have a natural deduction proof of $\theta$ from the assumptions 
\begin{equation}\label{Ass}
 \{A_1,\ldots,A_k\}
\end{equation}
 with last rule $\exists$ E. Let $\mA$ and $X$ be such that  $\mA\models _X A_i$, for $1\le i\le k$.
 
By the assumption, we have shorter proofs of a formula of the form $\exists x \phi$  from the assumptions \eqref{Ass} and of $\theta$ from 
\[\{\phi, A_{i_1},\ldots,A_{i_l}\},\]
where  $\{A_{i_1},\ldots,A_{i_l}\}\subseteq  \{A_1,\ldots,A_k\}$.
Note that the variable $x$ cannot appear free in  $\theta$  and  in $A_{i_1},\ldots,A_{i_l}$. By the induction assumption, we get that  $\mA\models _X \exists x \phi$, hence 
\begin{equation}\label{A} 
\mA\models _{X(F/x)} \phi
\end{equation} 
for some $F\colon X\rightarrow A$. Since $x$ does not appear free in the formulas $A_{i_j}$, Proposition \ref{freevar} implies that    
\begin{equation}\label{BBB}
\mA\models _{X(F/x)} A_{i_j}
\end{equation} 
 for $1\le j\le l$. By \eqref{A} and \eqref{BBB}, and   the induction assumption, we get that $\mA\models _{X(F/x)} \theta$ and, since $x$ does not appear free in $\theta$, it follows again by Proposition \ref{freevar} that $\mA\models _{X} \theta$.

\item[$\exists$ I]  


Assume that we have a natural deduction proof of $\exists x\psi$ from the assumptions 
\begin{equation*}
 \{A_1,\ldots,A_k\}
\end{equation*}
 with last rule $\exists$ I. Let $\mA$ and $X$ be such that  $\mA\models _X A_i$, for $1\le i\le k$. By the assumption, we have a shorter proof of  $\psi(t/x)$ from the same assumptions. By the induction assumption, we get that
 $\mA\models _X \psi(t/x)$. Lemma \ref{terms} now implies that $\mA\models _{X(F/x)} \psi(x)$, where 
 $F(s)= t^{\mA}\langle s\rangle$. Therefore, we get $\mA\models _X \exists x\psi$.

\item[ $\forall$ E]

Assume that we have a natural deduction proof of $\psi(t/x)$ from the assumptions 
\begin{equation*}
\{A_1,\ldots,A_k\}
\end{equation*}
with last rule $\forall$ E. Let $\mA$ and $X$ be such that  $\mA\models _X A_i$, for $1\le i\le k$. By the assumption, we have a shorter proof of  $\forall x\psi$ from the same assumptions. By the induction assumption, we get that
$\mA\models _X \forall x\psi$ and hence 
\begin{equation}\label{firststep}
\mA\models _{X(A/x)} \psi(x). 
\end{equation}
We need to show  $\mA\models _X\psi(t/x)$. We can use  Lemma \ref{terms} to show this: by Lemma \ref{terms}, it suffices to show that  $\mA\models _{X(F/x)}\psi(x)$. But now obviously $ X(F/x)\subseteq  X(A/x)$, hence   $\mA\models _{X(F/x)}\psi(x)$ follows using \eqref{firststep} and   Proposition \ref{Downward closure}.


\item[$\forall$ I] 
 Assume that we have a natural deduction proof of $\forall x\psi$ from the assumptions 
\begin{equation*}
 \{A_1,\ldots,A_k\}
\end{equation*}
 with last rule $\forall$ I. Let $\mA$ and $X$ be such that  $\mA\models _X A_i$, for $1\le i\le k$. By the assumption, we have a shorter proof of  $\psi$ from the same assumptions. Note that the variable $x$ cannot appear free in the $A_i$'s and hence by Proposition \ref{freevar}  
\[\mA\models _{X(A/x)} A_i,\]
for $1\le i\le k$. By the induction assumption, we  get that $\mA\models _{X(A/x)}\psi$, and finally that $\mA\models _{X} \forall x\psi$ as wanted.

\item[Rule 7] Assume that we have a natural deduction proof of $\forall y\exists x (\dep(\vec{z},x)\wedge \phi)$ from the assumptions
 $\{A_1,\ldots,A_k\}$ with last rule 7. Let $\mA$ and $X$ be such that  $\mA\models _X A_i$, for $1\le i\le k$. By the assumption, we have a shorter proof of $\exists x \forall y \phi$  from the assumptions $\{A_1,\ldots,A_k\}$ and thus by the induction assumption we get 
\[\mA\models _X \exists x \forall y \phi.\]
By Proposition \ref{freevar}, it follows that
\[\mA\models _{X \upharpoonright (\Fr(\phi)- \{x,y\})} \exists x \forall y \phi.\]
Hence there is $F\colon X \upharpoonright (\Fr(\phi)- \{x,y\})\rightarrow A$ such that 
 \[\mA\models _{X \upharpoonright (\Fr(\phi)- \{x,y\})(FA/xy)}\phi. \]
By the definition $F$, we have 
 \[\mA\models _{X \upharpoonright  (\Fr(\phi)- \{x,y\})(FA/xy)} \dep(\vec{z},x)\wedge \phi, \]
 where $\vec{z}$ lists the variables in $\Fr(\phi)-\{x,y\}$.  By redefining $F$ as a function with domain 
\[X \upharpoonright (\Fr(\phi)-\{x,y\})(A/y),\]
 it follows that 
 \[\mA\models _{(X \upharpoonright (\Fr(\phi)-\{x,y\}))(A/y)} \exists x (\dep(\vec{z},x)\wedge \phi), \]
and finally that 
\[\mA\models _{X \upharpoonright(\Fr(\phi)- \{x,y\})} \forall y\exists x (\dep(\vec{z},x)\wedge \phi). \]
By Proposition \ref{freevar}, we may conclude that
\[\mA\models _{X} \forall y\exists x (\dep(\vec{z},x)\wedge \phi). \]
In fact, it is straightforward to show that this rule is based on the corresponding  logical equivalence:
\[  \exists x \forall y\psi\equiv  \forall y\exists x (\dep(\vec{z},x)\wedge \phi).  \]

\item[Rule 8]   Assume that we have a natural deduction proof of  $\psi$ of the form
\begin{equation*}
\begin{array}{l}
 \forall\vec{x_0}\exists\vec{y_0}(B(\vec{x_0},\vec{y_0})\wedge\\
\forall\vec{x_1}\exists\vec{y_1}(B(\vec{x_1},\vec{y_1})\wedge
\bigwedge_{=\!(\vec{w}_0^p,y_{0,p})\in S}(\vec{w}_0 ^p=\vec{w}_1 ^p\to y_{0,p}=y_{1,p})))
\end{array} 
\end{equation*}
 from the assumptions $ \{A_1,\ldots,A_k\}$ with last rule 8. 
Let $\mA$ and $X$ be such that  $\mA\models _X A_i$, for $1\le i\le k$. By the assumption, we have a shorter proof of $\phi$
\begin{equation}\label{C}
\phi:= \forall\vec{x_0}\exists\vec{y_0} (\bigwedge_{1\le j\le k}\dep(\vec{w}^{i_j},y_{0,{i_j}})\wedge B(\vec{x_0},\vec{y_0})),
\end{equation}
from the same assumptions. By the induction assumption, we get that  
$\mA\models_X\phi$, hence there are functions $F_{0,r}$, $1\le r\le n$, such that
\begin{equation}\label{Bwise0}
\mA\models_{X(A\cdots AF_{0,1}\cdots F_{0,n}/\vec{x_0},\vec{y_0})} \bigwedge_{1\le j\le k}\dep(\vec{w}^{i_j},y_{0,{i_j}})\wedge B(\vec{x_0},\vec{y_0}). 
\end{equation}
We can now interpret  the variable $y_{1,r}$  essentially  by the same function $F_{0,r}$ that was used to interpret  $y_{0,r}$. Suppose that there is $1\le j\le k$ such that $y_{0,r}=y_{0,i_j}$.
For the sake of bookkeeping, we write $\vec{w}^{i_j}$ as  $\vec{w}^{i_j}_0$ and by $\vec{w}^{i_j}_1$ 
we denote the tuple arising from  $\vec{w}^{i_j}$  by replacing $x_{0,s}$ by $x_{1,s}$ and $y_{0,s}$ by $y_{1,s}$, respectively. We can now  define $F_{1,r}$ such that $F_{1,r}(s):=s'(y_{0,r})$, where $s'$ is any assignment satisfying  $s'( \vec{w}^{i_j}_0) =s( \vec{w}^{i_j}_1)$ ($s$ and $s'$ are applied pointwise).
 In the case  there is no $1\le j\le k$ such that $y_{0,r}=y_{0,i_j}$, we use the tuple $x_{0,1},\ldots,x_{0,m}$, instead of  $\vec{w}^{i_j}$, and proceed analogously.

 We first show that 
\begin{equation}\label{Bwise2}
\mA\models_{X(\bar{A}\bar{F_0}\bar{A}\bar{F_1}/\vec{x_0},\vec{y_0}\vec{x_1},\vec{y_1})}  B(\vec{x_1},\vec{y_1})
\end{equation}
holds.  The variables in $\vec{x_0}$ and $\vec{y_0}$ do not appear in  $B(\vec{x_1},\vec{y_1})$, thus \eqref{Bwise2} holds
iff 
\begin{equation}\label{Bwise3}
\mA\models_{X(\bar{A}\bar{F_1}/\vec{x_1},\vec{y_1})}  B(\vec{x_1},\vec{y_1}). 
\end{equation}
Now  \eqref{Bwise3} is equivalent to the truth of the second conjunct in  \eqref{Bwise0},  modulo renaming  (in the team and in the formula)  the variables $x_{0,i}$ and $y_{0,i}$ by $x_{1,i}$ and $y_{1,i}$, respectively. Hence \eqref{Bwise2} follows. 

Let us then show that
\begin{equation}\label{Bwise}
\mA\models_{X(\bar{A}\bar{F_0}\bar{A}\bar{F_1}/\vec{x_0},\vec{y_0}\vec{x_1},\vec{y_1})}  \bigwedge_{=\!(\vec{w}^p,y_{0,p})\in S}(\vec{w}_0 ^p=\vec{w}_1 ^p\to y_{0,p}=y_{1,p}) .
\end{equation}
Let $=\!(\vec{w}^p,y_{0,p})\in S$.
We need to show that the formula  
\[\vec{w}_0 ^p=\vec{w}_1 ^p\to y_{0,p}=y_{1,p},\]
i.e., the formula
\[\neg (\vec{w}_0 ^p=\vec{w}_1 ^p)\vee y_{0,p}=y_{1,p}\]
is satisfied by the team in \eqref{Bwise}. Since this formula is first-order, it suffices to show the claim for every assignment  $s$ in the team. But this is immediate since, assuming $s(\vec{w}_0 ^p)=s(\vec{w}_1 ^p)$, we get by the definition of $F_{1,p}$ that $s(y_{0,p})=s(y_{1,p})$. Now by combining  \eqref{Bwise} and \eqref{Bwise2} with \eqref{Bwise0},  we get that $\mA\models _X \psi$ as wanted.
\end{itemize}
\end{proof}

\section{The Completeness Theorem}

In this section we show that our proof system allows us to derive all first-order consequences of sentences of dependence logic.

\subsection{The roadmap for the proof}\label{roadmap}

Our method for finding  an explicit axiomatization is based on an idea of Jon Barwise \cite{MR0465788}. Instead of dependence logic, Barwise considers the related concept of partially ordered quantifier-prefixes. The roadmap to establishing  that the axioms are sufficiently strong goes as follows:

\begin{enumerate}
\item We will first show that from any sentence $\phi$ it is possible to derive a logically equivalent sentence $\phi'$ that is of the special form 
\begin{equation*}
\forall\vec{x_0}\exists\vec{y_0} (\bigwedge_{1\le j\le k}\dep(\vec{w}^{i_j},y_{0,{i_j}})\wedge \psi(\vec{x_0},\vec{y_0})),
\end{equation*}
where $\psi$ is quantifier-free formula without dependence atoms. 

\item The sentence $\phi'$ above can be shown to be equivalent, in countable models, to the game expression $\Phi$.
 
\begin{equation*}
\begin{array}{l}
\forall\vec{x_0}\exists\vec{y_0}(\psi(\vec{x_0},\vec{y_0})\wedge\\
\forall\vec{x_1}\exists\vec{y_1}(\psi(\vec{x_1},\vec{y_1})\wedge
\bigwedge_{=\!(\vec{w}_0^p,y_{0,p})\in S}(\vec{w}_0 ^p=\vec{w}_1 ^p\to y_{0,p}=y_{1,p})\wedge\\
\forall\vec{x_2}\exists\vec{y_2}(\psi(\vec{x_2},\vec{y_2})\wedge
\bigwedge_{=\!(\vec{w}_0^p,y_{0,p})\in S}(\vec{w}_1 ^p=\vec{w}_2 ^p\to y_{1,p}=y_{2,p})\wedge\\
\hspace{2.75cm}\wedge
\bigwedge_{=\!(\vec{w}_0^p,y_{0,p})\in S}(\vec{w}_0 ^p=\vec{w}_2 ^p\to y_{0,p}=y_{2,p})\wedge\\
...\\
...\\
\hspace{2cm}...)))\\
\end{array} 
\end{equation*}


\item The game expression $\Phi$ can be approximated by the first-order formulas $\Phi^n$ (note that the rule 8 applied to $\phi'$ gives exactly  $\Phi^2$ ):

\[
\begin{array}{l}
\forall\vec{x_0}\exists\vec{y_0}(\psi(\vec{x_0},\vec{y_0})\wedge\\
\forall\vec{x_1}\exists\vec{y_1}(\psi(\vec{x_1},\vec{y_1})\wedge
\bigwedge_{=\!(\vec{w}_0^p,y_{0,p})\in S}(\vec{w}_0 ^p=\vec{w}_1 ^p\to y_{0,p}=y_{1,p})\wedge\\
\forall\vec{x_2}\exists\vec{y_2}(\psi(\vec{x_2},\vec{y_2})\wedge
\bigwedge_{=\!(\vec{w}_0^p,y_{0,p})\in S}(\vec{w}_1 ^p=\vec{w}_2 ^p\to y_{1,p}=y_{2,p})\wedge\\
\hspace{2.75cm}\wedge
\bigwedge_{=\!(\vec{w}_0^p,y_{0,p})\in S}(\vec{w}_0 ^p=\vec{w}_2 ^p\to y_{0,p}=y_{2,p})\wedge\\
...\\
...\\
\forall\vec{x}_{n-1}\exists\vec{y}_{n-1}(\psi(\vec{x}_{n-1},\vec{y}_{n-1})\wedge  \bigwedge_{0\le i< n-1} \bigwedge_{=\!(\vec{w}_0^p,y_{0,p})\in S}(\vec{w}_i ^p=\vec{w}_{n-1} ^p\\
\hspace{7cm}\to y_{i,p}=y_{n-1,p})    )\cdots)\\
\end{array} 
\]


\item Then we show that from the sentence $\phi'$ of Step 1 it is possible to derive the above approximations $\Phi^n$.

\item We then note that if $\mA$
  is a countable recursively saturated (or finite) model, then \[\mA\models\Phi\leftrightarrow\bigwedge_n\Phi^n.\]

\item Finally, we show that  for any $T\subseteq\D$ and $\phi\in \FO$:
 $$T\models\phi\iff T\vdashd\phi$$ as follows: For the non-trivial direction, suppose  $T\not\vdashd\phi$. Let $T^*$ consist of all the approximations of the dependence sentences in $T$.  Now $T^*\cup\{\neg\phi\}$ is deductively consistent in first order logic, and has therefore a  countable recursively saturated model $\mA$. But then $\mA\models T\cup\{\neg\phi\}$, so $T\not\models\phi$.
\end{enumerate}

\subsection{From $\phi$ to $\phi'$ in normal form}

In this section we show that  from any sentence $\phi$ it is possible to derive a logically equivalent sentence $\phi'$  of the special form 
\begin{equation}\label{1}
\forall x_1\ldots \forall x_m \exists y_{1}\ldots \exists y_{n}  (\bigwedge_{1\le j\le k}\dep(\vec{w}^{i_j},y_{i_j})\wedge \theta(\vec{x},\vec{y}))
\end{equation}
where  $\theta$ is quantifier-free formula without dependence atoms. 

\begin{proposition}Let $\phi$ be a sentence of dependence logic. Then $\phi\vdashd \phi'$, where $\phi'$ is of the form \eqref{1}, and $\phi'$ is logically equivalent to $\phi$. 
\end{proposition}

\begin{proof}
We will establish the claim in several steps. Without loss of generality, we assume that in $\phi$ each variable is quantified only once and that, in  the dependence atoms of $\phi$, only variables (i.e. no complex terms) occur. 

\begin{itemize}
\item{Step 1.} We derive from $\phi$ an equivalent sentence in prenex normal form:
\begin{equation}\label{2}
  Q^1 x_1\ldots Q^mx_m\theta,
\end{equation}
where $Q^i\in \{ \exists,\forall \}$ and $\theta$ is a quantifier-free formula. 

We will  prove the claim for every formula $\phi$  satisfying the assumptions made in the beginning of the proof and the assumption (if $\phi $ has free variables) that no variable appears both free and bound in $\phi$.  It suffices to consider the case $\phi:= \psi\vee \theta$, since the  case of conjunction is analogous and the other cases are trivial.

 By the induction assumption, we have derivations  $ \psi\vdashd \psi^*$ and $\theta\vdashd \theta^* $, where
\begin{eqnarray*}
\psi^* &=& Q^1 x_1\ldots Q^m x_m\psi_0,\\
\theta^* &= & Q^{m+1} x_{m+1}\ldots Q^{m+n}x_{m+n}\theta_0,
\end{eqnarray*}
 and  $\psi\equiv \psi^*$ and $\theta\equiv \theta^* $.  Now $\phi\vdashd \psi^* \vee\theta^*$, using two 
applications of the rule \ref{rule1}. Next we prove using induction on 
$m$ that, from $\psi^* \vee\theta^*$, we can derive 
\begin{equation}\label{goalform}
 Q^1 x_1\ldots Q^mx_mQ^{m+1} x_{m+1}\ldots Q^{m+n}x_{m+n}(\psi_0\vee \theta_0).
\end{equation}
Let $m=0$. We prove this case again by induction; for $n=0$ the claim holds. Suppose that $n=l+1$. We assume that $Q^1=\exists$. The  case $Q^1=\forall$ is analogous. The following deduction now shows the claim:
\begin{enumerate}
\item  $\psi_0 \vee Q^{1} x_{1}\ldots Q^{n}x_{n}\theta_0$ 
\item  $Q^{1} x_{1}\ldots Q^{n}x_{n}\theta_0\vee \psi_0 $ (rule \ref{rule2})
\item $Q^{1} x_{1}(Q^{2} x_{2}\ldots Q^{n}x_{n}\theta_0\vee \psi_0) $ (rule \ref{rule4})
\item $Q^{1}x_{1}\ldots Q^{n}x_{n}( \psi_0\vee \theta_0)$ ($\exists$ E and D1),  
\end{enumerate}
where D1 is the derivation 
\begin{enumerate}
\item $Q^{2} x_{2}\ldots Q^{n}x_{n}\theta_0\vee \psi_0$
\item      \hspace{1cm}     .
\item        \hspace{1cm}    .
\item       \hspace{1cm}     .
\item $ Q^{2} x_{2}\ldots Q^{n}x_{n}(\theta_0\vee \psi_0)$ (induction assumption)
\item      \hspace{1cm}     .
\item        \hspace{1cm}    .
\item       \hspace{1cm}     .
\item $ Q^{2} x_{2}\ldots Q^{n}x_{n}(\psi_0\vee \theta_0)$ (D2)
\item $Q^{1}x_{1}\ldots Q^{n}x_{n}(\psi_0\vee \theta_0)$ ($\exists$ I)
\end{enumerate}
where D2 is a derivation that swaps the disjuncts. This concludes the proof for the case $m=0$.

Assume then that $m=k+1$ and that the claim holds for $k$. Now the following derivation shows the claim: (assume $Q^1=\exists$)

\begin{enumerate}
\item  $Q^1x_1Q^2 x_2\ldots Q^mx_m\psi_0 \vee  Q^{m+1}x_{m+1}\ldots Q^{m+n}x_{m+n}\theta_0$
\item  $Q^1x_1(   Q^2 x_2\ldots Q^mx_m\psi_0 \vee  Q^{m+1}x_{m+1}\ldots Q^{m+n}x_{m+n}\theta_0)$ (rule \ref{rule4})
\item $Q^1 x_1\ldots Q^mx_mQ^{m+1} x_{m+1}\ldots Q^{m+n}x_{m+n}(\psi_0\vee\theta_0)$ ($\exists$ E and D3)
\end{enumerate}

where D3 is 
\begin{enumerate}
\item  $Q^2 x_2\ldots Q^mx_m\psi_0 \vee  Q^{m+1}x_{m+1}\ldots Q^{m+n}x_{m+n}\theta_0$
\item      \hspace{1cm}      .
\item       \hspace{1cm}     .
\item         \hspace{1cm}   .
\item $ Q^2 x_2\ldots Q^mx_mQ^{m+1} x_{m+1}\ldots Q^{m+n}x_{m+n}(\psi_0\vee\theta_0)$ (ind. assumption)
\item  $Q^1 x_1\ldots Q^mx_mQ^{m+1} x_{m+1}\ldots Q^{m+n}x_{m+n}(\psi_0\vee\theta_0)$ ($\exists$ I)
\end{enumerate}
This concludes the proof.

\item{Step 2.} Next we show that from a quantifier-free formula $\theta$ it is possible to derive an equivalent formula of the form: 
\begin{equation}\label{3}
\exists z_1\ldots \exists z_n(\bigwedge_{1\le j\le n}\dep(\vec{x}^j,z_j)\wedge \theta^*),   
\end{equation}
where $\theta^*$ is a quantifier-free formula without dependence atoms. Again we prove the claim using induction on $\theta$. If $\theta$ is  first-order atomic or negated atomic, then the claim holds. If  
$\theta$ is of the form  $\dep(\vec{y},x)$, then rule \ref{rule5} allows us to derive 
\[\exists z( \dep(\vec{y},z)\wedge z=x)\]
as wanted.

Assume then that $\theta := \phi \vee \psi$. By the induction assumption,  we have derivations  $ \phi\vdashd \phi^*$ and $\psi\vdashd \psi^* $, where 
\begin{eqnarray*}
\phi^* &=& \exists y_1\ldots \exists y_n(\bigwedge_{1\le j\le n}\dep(\vec{z}^j,y_j)\wedge \phi_0),\\
\psi^* &=& \exists y_{n+1}\ldots \exists y_{n+m}(\bigwedge_{n+1\le j\le n+m}\dep(\vec{z}^j,y_j)\wedge \psi_0)
\end{eqnarray*}
 such that   $\phi\equiv \phi^*$,   $\psi\equiv \psi^*$,  and  $\phi_0$ and $\psi_0$ are quantifier-free formulas without dependence atoms, and $y_i$, for $1\le i \le n$, does not appear in $\psi^*$ and $y_i$, for $n+1\le i \le n+m$, does not appear in $\phi^*$. Now $\theta\vdashd \psi^* \vee\theta^*$, using two 
applications of the rule \ref{rule1} and rule \ref{rule6} allows us to derive 
\[ \exists y_1\ldots \exists y_n \exists y_{n+1} \ldots \exists y_{n+m}(\bigwedge_{1\le j\le n+m}\dep(\vec{z}^j,y_j)\wedge (\phi_0 \vee \psi_0)) \]
which is now equivalent to $\theta$ and has the required form. 
Note that in the case  $\theta := \phi \wedge \psi$  only first-order inference rules for conjunction and $\exists$ are needed and it is similar to the proof of Step 1.

\item{Step 3.} The deductions  in Step 1 and 2 can be combined  (from $\phi$ to \eqref{2}, and then from $\theta$ to  \eqref{3}) to show that  
\begin{equation}\label{4}
\phi\vdashd Q^1 x_1\ldots Q^mx_m \exists z_1\ldots \exists z_n(\bigwedge_{1\le j\le n}\dep(\vec{x}^j,z_j)\wedge \theta^*).   
\end{equation}

\item{Step 4.} We transform the $Q$-quantifier prefix in \eqref{4} to $\forall^*\exists ^*$-form by using rule \ref{rule7} and pushing the new dependence atoms as new conjuncts to 
\begin{equation}\label{5}
\bigwedge_{1\le j\le n}\dep(\vec{x}^j,z_j).
\end{equation}
Note that each swap of the quantifier $\exists x_j$ with a universal quantifier gives rise to a new dependence atom   $\dep(\vec{x}_i,x_j)$ which we can then push to the quantifier-free part of the formula.

We prove the claim using induction on the length $m$ of the $Q$-quantifier block in \eqref{4}. For $m=1$ the claim holds. Suppose that the claim holds for $k$ and $m=k+1$. Assume $Q_1=\forall$. Then the following derivation can be used:

\begin{enumerate}
\item$ \forall x_1 Q^2x_2\ldots Q^mx_m \exists z_1\ldots \exists z_n(\bigwedge_{1\le j\le n}\dep(\vec{x}^j,z_j)\wedge \theta^*)$ 
\item $ Q^2 x_2\ldots Q^mx_m \exists z_1\ldots \exists z_n(\bigwedge_{1\le j\le n}\dep(\vec{x}^j,z_j)\wedge \theta^*)$ ($\forall$ E)
\item       \hspace{1cm}       .
\item       \hspace{1cm}       .
\item       \hspace{1cm}       .
\item $\forall x_{i_1}\cdots \forall  x_{i_h}\exists \vec{x}'  \exists \vec{z} (\bigwedge_{1\le j\le n'}\dep(\vec{x}^j,w_j)\wedge \theta^*)$ (ind. assumption)
\item  $\forall x_1\forall x_{i_1}\cdots \forall  x_{i_h}\exists \vec{x}'  \exists \vec{z} (\bigwedge_{1\le j\le n'}\dep(\vec{x}^j,w_j)\wedge \theta^*)$ ($\forall$ I)
\end{enumerate}
This concludes the proof in the case  $Q_1=\forall$.
Suppose then that $Q_1= \exists$ and that $Q_i=\forall$ at least for some $i\ge 2$.
Now the following derivation can be used:

\begin{enumerate}
\item$ \exists x_1 Q^2x_2 \ldots Q^mx_m \exists z_1\ldots \exists z_n(\bigwedge_{1\le j\le n}\dep(\vec{x}^j,z_j)\wedge \theta^*)$ 
\item  $\exists x_1\forall x_{i_1}\cdots \forall  x_{i_h}\exists \vec{x}'  \exists \vec{z} (\bigwedge_{1\le j\le n'}\dep(\vec{x}^j,w_j)\wedge \theta^*)$ ($\exists$ E and D4)

\item $\forall x_{i_1} \exists x_1(\dep(x_1)\wedge\forall x_{i_2}\cdots \forall  x_{i_h}\exists \vec{x}'  \exists \vec{z} (\bigwedge_{1\le j\le n'}\dep(\vec{x}^j,w_j)\wedge \theta^*)$  (rule \ref{rule7})
\item   \hspace{1cm}           .
\item   \hspace{1cm}          .
\item  \hspace{1cm}           .
\item\label{line7}   $\forall x_{i_1} \exists x_1\forall x_{i_2}\cdots \forall  x_{i_h}\exists \vec{x}'  \exists \vec{z} (\bigwedge_{1\le j\le n'+1}\dep(\vec{x}^j,w_j)\wedge \theta^*)$ (D5)
\item  $\exists x_1\forall x_{i_2}\cdots \forall  x_{i_h}\exists \vec{x}'  \exists \vec{z} (\bigwedge_{1\le j\le n'+1}\dep(\vec{x}^j,w_j)\wedge \theta^*)$  ($\forall$ E)
\item    \hspace{1cm}         .
\item   \hspace{1cm}          .
\item   \hspace{1cm}           .
\item $\forall x_{i_2}\cdots \forall  x_{i_h}\exists x_1\exists \vec{x}'  \exists \vec{z} (\bigwedge_{1\le j\le n''}\dep(\vec{x}^j,w_j)\wedge \theta^*)$ (induction as.)
\item $\forall x_{i_1}\forall x_{i_2}\cdots \forall  x_{i_h}\exists x_1\exists \vec{x}'  \exists \vec{z} (\bigwedge_{1\le j\le n''}\dep(\vec{x}^j,w_j)\wedge \theta^*)$ ($\forall$ I)
\end{enumerate}
where, on line \ref{line7}, $\dep(\vec{x}^j,w_j)$ is $\dep(x_1)$ for  $j=n'+1$. Furthermore, above  D4 refers to  the following deduction
\begin{enumerate}
\item $ Q^2 x_2\ldots Q^mx_m \exists z_1\ldots \exists z_n(\bigwedge_{1\le j\le n}\dep(\vec{x}^j,z_j)\wedge \theta^*)$ 
\item  \hspace{1cm}           .
\item  \hspace{1cm}           .
\item  \hspace{1cm}           .
\item $\forall x_{i_1}\cdots \forall  x_{i_h}\exists \vec{x}'  \exists \vec{z} (\bigwedge_{1\le j\le n'}\dep(\vec{x}^j,w_j)\wedge \theta^*)$ (ind. assumption)
\item  $\exists x_1\forall x_{i_1}\cdots \forall  x_{i_h}\exists \vec{x}'  \exists \vec{z} (\bigwedge_{1\le j\le n'}\dep(\vec{x}^j,w_j)\wedge \theta^*)$ ($\exists$ I)
\end{enumerate}
 and D5 is a straightforward deduction that is easy to construct. This concludes the proof of the case $Q_1=\exists$ and also of Step 4.
\end{itemize}
Steps 1-4 show that from a sentence $\phi$ a sentence of the form 
\begin{equation}\label{6}
\forall x_1\ldots \forall x_m \exists y_{1}\ldots \exists y_{n}  (\bigwedge_{1\le j\le k}\dep(\vec{w}^{i_j},y_{i_j})\wedge \theta(\vec{x},\vec{y}))
\end{equation}
can be deduced.  Furthermore, $\phi$ and the sentence in \eqref{6} are logically equivalent since
 logical equivalence is  preserved in each of the Steps 1-4.
\end{proof}

\subsection{Derivation of the approximations $\Phi^n$}

In the previous section we showed that from any sentence of dependence logic a logically equivalent sentence of the form 
\begin{equation}\label{normalform}
\forall\vec{x_0}\exists\vec{y_0} (\bigwedge_{1\le j\le k}\dep(\vec{w}^{i_j},y_{0,{i_j}})\wedge \psi(\vec{x_0},\vec{y_0})),
\end{equation}
can be derived,  where $\psi$ is quantifier-free formula without dependence atoms. We will next show that the approximations $\Phi^n$ of the  game expression  $\Phi$ (discussed in  Section \ref{roadmap}) correponding to sentence \eqref{normalform}, can be deduced from it. 

The  formulas $\Phi$ and $\Phi^n$ are  defined as follows. 
\begin{definition}\label{approximation}
Let $\phi$ be the formula \eqref{normalform}, where
$\vec{x_0}=(x_{0,1},\ldots,x_{0,m})$ and $\vec{y_0}=(y_{0,1},\ldots,y_{0,n})$, and  
 the variables in $\vec{w}^{i_j}$ are contained in the set
\[  \{x_{0,1},\ldots,x_{0,m},y_{0,1},\ldots,y_{0,i_j-1} \}  .\]
Furthermore, let $S$ be the set   containing the conjuncts of 
\[\bigwedge_{1\le j\le k}\dep(\vec{w}^{i_j},y_{0,{i_j}}), \]
and  the dependence atom $\dep(x_{0,1},\ldots,x_{0,m},y_{0,p})$ for each of the variables $y_{0,p}$ ($1\le p \le n$) such that $y_{0,p} \notin \{y_{0,i_1},\ldots, y_{0,i_k}\}$. Define $\vec{x_l}=(x_{l,1},\ldots,x_{l,m})$, $\vec{y_l}=(y_{l,1},\ldots,y_{l,n})$ and  $\vec{w}_l^p$ analogously. 

\begin{itemize}
\item The infinitary formula   $\Phi$ is now defined as:
 \begin{equation}\label{gameformula}
\begin{array}{l}
\forall\vec{x_0}\exists\vec{y_0}(\psi(\vec{x_0},\vec{y_0})\wedge\\
\forall\vec{x_1}\exists\vec{y_1}(\psi(\vec{x_1},\vec{y_1})\wedge
\bigwedge_{=\!(\vec{w}_0^p,y_{0,p})\in S}(\vec{w}_0 ^p=\vec{w}_1 ^p\to y_{0,p}=y_{1,p})\wedge\\
\forall\vec{x_2}\exists\vec{y_2}(\psi(\vec{x_2},\vec{y_2})\wedge
\bigwedge_{=\!(\vec{w}_0^p,y_{0,p})\in S}(\vec{w}_1 ^p=\vec{w}_2 ^p\to y_{1,p}=y_{2,p})\wedge\\
\hspace{2.75cm}\wedge
\bigwedge_{=\!(\vec{w}_0^p,y_{0,p})\in S}(\vec{w}_0 ^p=\vec{w}_2 ^p\to y_{0,p}=y_{2,p})\wedge\\
...\\
...\\
\hspace{2cm}...)))\\
\end{array} 
\end{equation}

\item The $n$th approximation $\Phi^n$ of $\phi$  is defined as: 

\[
\begin{array}{l}
\forall\vec{x_0}\exists\vec{y_0}(\psi(\vec{x_0},\vec{y_0})\wedge\\
\forall\vec{x_1}\exists\vec{y_1}(\psi(\vec{x_1},\vec{y_1})\wedge
\bigwedge_{=\!(\vec{w}_0^p,y_{0,p})\in S}(\vec{w}_0 ^p=\vec{w}_1 ^p\to y_{0,p}=y_{1,p})\wedge\\
\forall\vec{x_2}\exists\vec{y_2}(\psi(\vec{x_2},\vec{y_2})\wedge
\bigwedge_{=\!(\vec{w}_0^p,y_{0,p})\in S}(\vec{w}_1 ^p=\vec{w}_2 ^p\to y_{1,p}=y_{2,p})\wedge\\
\hspace{2.75cm}\wedge
\bigwedge_{=\!(\vec{w}_0^p,y_{0,p})\in S}(\vec{w}_0 ^p=\vec{w}_2 ^p\to y_{0,p}=y_{2,p})\wedge\\
...\\
...\\
\forall\vec{x}_{n-1}\exists\vec{y}_{n-1}(\psi(\vec{x}_{n-1},\vec{y}_{n-1})\wedge  \bigwedge_{0\le i< n-1} \bigwedge_{=\!(\vec{w}_0^p,y_{0,p})\in S}(\vec{w}_i ^p=\vec{w}_{n-1} ^p\\
\hspace{7cm}\to y_{i,p}=y_{n-1,p})    )\cdots)\\
\end{array} 
\]
\end{itemize}
\end{definition}
We will next show that the approximations  $\Phi^n$ can be deduced from $\phi$.

\begin{proposition}\label{Prop16} Let $\phi$ and   $\Phi^n$ be as in Definition \ref{approximation}. Then 
\[ \phi \vdashd  \Phi^n \]
for all $n\ge 1$.
\end{proposition}

\begin{proof}
We will prove a slightly stronger claim: $\phi \vdashd  \Omega^n$, where   $\Omega^n$ is defined otherwise as 
 $\Phi^n$, except that,  on the last line, we also have the conjunct   $$\bigwedge_{1\le j\le k}\dep(\vec{w}_{n-1}^{i_j},y_{n-1,i_j})$$  (see \eqref{normalform}), with the variables $x_{0,l}$ and $y_{0,l}$ replaced by  
$x_{n-1,l}$ and $y_{n-1,l}$, respectively.

Let us first note that  clearly
\[\Omega^n \vdashd \Phi^n,\] 
since this amounts only to showing that the new conjunct can be eliminated from the formula. 
Hence to prove the proposition, it suffices to  show that 
\[ \phi \vdashd   \Omega^n \]
for all $n\ge 1$. We will prove the claim using induction on $n$. The claim holds for $n=1$, 
since $\Omega^1 =  \phi$. Suppose that $n=h+1$. By the induction assumption, $\phi \vdashd   \Omega^h$. Let us  recall that 
$\Omega^h$  is the sentence
\[
\begin{array}{l}
\forall\vec{x_0}\exists\vec{y_0}(\psi(\vec{x_0},\vec{y_0})\wedge\\
\forall\vec{x_1}\exists\vec{y_1}(\psi(\vec{x_1},\vec{y_1})\wedge
\bigwedge_{=\!(\vec{w}_0^p,y_{0,p})\in S}(\vec{w}_0 ^p=\vec{w}_1 ^p\to y_{0,p}=y_{1,p})\wedge\\
\forall\vec{x_2}\exists\vec{y_2}(\psi(\vec{x_2},\vec{y_2})\wedge
\bigwedge_{=\!(\vec{w}_0^p,y_{0,p})\in S}(\vec{w}_1 ^p=\vec{w}_2 ^p\to y_{1,p}=y_{2,p})\wedge\\
\hspace{2.75cm}\wedge
\bigwedge_{=\!(\vec{w}_0^p,y_{0,p})\in S}(\vec{w}_0 ^p=\vec{w}_2 ^p\to y_{0,p}=y_{2,p})\wedge\\
...\\
...\\
\forall\vec{x}_{h-1}\exists\vec{y}_{h-1}(\psi(\vec{x}_{h-1},\vec{y}_{h-1})\wedge \bigwedge_{1\le j\le k}\dep(\vec{w}_{h-1}^{i_j},y_{h-1,i_j})\wedge \\
 \hspace{2cm}\bigwedge_{0\le i< h-1} \bigwedge_{=\!(\vec{w}_0^p,y_{0,p})\in S}(\vec{w}_i ^p=\vec{w}_{h-1} ^p
\to y_{i,p}=y_{h-1,p})    )\cdots)\\
\end{array} 
\]
The claim  $\phi \vdashd   \Omega^{h+1}$ is now proved as follows:  
\begin{enumerate}
\item\label{subf} We first  eliminate the quantifiers and conjuncts of $\Omega^h$ to show that, from $\phi$,   the following  subformula of $\Omega^{h}$  can be deduced:

\[
\begin{array}{l}
\forall\vec{x}_{h-1}\exists\vec{y}_{h-1}(\psi(\vec{x}_{h-1},\vec{y}_{h-1})\wedge \bigwedge_{1\le j\le k}\dep(\vec{w}_{h-1}^{i_j},y_{h-1,i_j})\wedge \\
 \hspace{2cm}\bigwedge_{0\le i< h-1} \bigwedge_{=\!(\vec{w}_0^p,y_{0,p})\in S}(\vec{w}_i ^p=\vec{w}_{h-1} ^p
\to y_{i,p}=y_{h-1,p})    )\\
\end{array} 
\]
\item\label{newsubf} Then  rule 8 essentially allows us to deduce   
\[
\begin{array}{l}
\forall\vec{x}_{h-1}\exists\vec{y}_{h-1}(\psi(\vec{x}_{h-1},\vec{y}_{h-1})\wedge \bigwedge_{0\le i< h-1} \bigwedge_{=\!(\vec{w}_0^p,y_{0,p})\in S}(\vec{w}_i ^p=\vec{w}_{h-1} ^p\\
\hspace{7,5cm} \to y_{i,p}=y_{h-1,p}) \wedge \\
\forall\vec{x}_{h}\exists\vec{y}_{h}(\psi(\vec{x}_{h},\vec{y}_{h})\wedge \bigwedge_{1\le j\le k}\dep(\vec{w}_{h}^{i_j},y_{h,i_j})\wedge  \\
\hspace{3,1cm} \bigwedge_{0\le i< h} \bigwedge_{=\!(\vec{w}_0^p,y_{0,p})\in S}(\vec{w}_i ^p=\vec{w}_{h} ^p
\to y_{i,p}=y_{h,p}) ))
\end{array} 
\]
\item and finally it suffices to "reverse" the derivation in Step 1 to show that  $\phi \vdashd   \Omega^{h+1}$. 
\end{enumerate}

We will now show the deduction of the formula in Step 2 assuming the formula in Step 1. To simply notation, we use the following shorthands:
\begin{itemize}
\item  $P_l:=\psi(\vec{x}_{l},\vec{y}_{l})$, 
\item $D_l:=\bigwedge_{1\le j\le k}\dep(\vec{w}_{l}^{i_j},y_{l,i_j})$,  
\item $C_l:= \bigwedge_{0\le i< l} \bigwedge_{=\!(\vec{w}_0^p,y_{0,p})\in S}(\vec{w}_i ^p=\vec{w}_{l} ^p
\to y_{i,p}=y_{l,p})$
\item $C^-_l:=\bigwedge_{0\le i< l-1} \bigwedge_{=\!(\vec{w}_0^p,y_{0,p})\in S}(\vec{w}_i ^p=\vec{w}_{l} ^p\to y_{i,p}=y_{l,p})$
\item $ C^+_l:= \bigwedge_{=\!(\vec{w}_0^p,y_{0,p})\in S}(\vec{w}_{l-1} ^p=\vec{w}_{l} ^p\to y_{l-1,p}=y_{l,p})$
\end{itemize}
It is important to note that $ C^-_l\wedge  C^+_l=C_l$. The following deduction now shows the claim:
\begin{enumerate}
\item $\forall\vec{x}_{h-1}\exists\vec{y}_{h-1}(  P_{h-1}\wedge D_{h-1}\wedge C_{h-1})$
\item \hspace{1cm}           .
\item     \hspace{1cm}        .
\item      \hspace{1cm}       .
\item $\forall\vec{x}_{h-1}\exists\vec{y}_{h-1}( D_{h-1}\wedge (P_{h-1}\wedge D_{h-1}\wedge C_{h-1}))$  (D1)
\item $\forall\vec{x}_{h-1}\exists\vec{y}_{h-1}((P_{h-1}\wedge D_{h-1}\wedge C_{h-1}) \wedge  \forall\vec{x}_{h}\exists\vec{y}_{h}((P_{h}\wedge D_{h}\wedge C^-_h)\wedge C^+_l)$ (Rule 8)
\item    \hspace{1cm}         .
\item     \hspace{1cm}        .
\item      \hspace{1cm}      .
\item $\forall\vec{x}_{h-1}\exists\vec{y}_{h-1}((P_{h-1}\wedge C_{h-1}) \wedge \forall\vec{x}_{h}\exists\vec{y}_{h}(P_{h}\wedge D_{h}\wedge C_h))$ (D2)
\end{enumerate}
In the above derivation, the derivation D1 just creates one extra copy of  $ D_{h-1}$, and D2 deletes   $ D_{h-1}$
from the formula and groups together $ C^-_h$ and $ C^+_h$. The deductions D1 and D2 can be easily constructed. 

This completes the proof Proposition \ref{Prop16}.
\end{proof}

\subsection{Back from approximations}

In this section we prove the main result of the paper.

The use of game expressions to analyze existential second order sentences is originally due to Lars Svenonius \cite{MR0209138}. Subsequently it was developed by Robert Vaught \cite{MR0409106}.

Basic fact about the game expressions is:

\begin{proposition}\label{17}
Let $\phi$ and $\Phi$ be as in Section 5.1. Then $\phi\models\Phi$. In countable models
 $\Phi\models\phi$.
\end{proposition}

\begin{proof}
Suppose $\phi$ is as in (\ref{normalform}) and $\mA\models\phi$. Suppose furthermore $\Phi$ as in (\ref{gameformula}). We show $\mA\models\Phi$.
By definition, the truth of $\Phi$ in $\mA$ means the existence of a winning strategy of player II in the following game $G(\mA,\Phi)$:

\begin{center}
$$
\begin{array}{c|ccccc}
I & \vec{a_0} &           & \vec{a_1} & \ldots\\
\hline
II &           &\vec{b_0} &          & \vec{b_1} & \ldots\\ 
\end{array}$$
\end{center}
where $\vec{a_i},\vec{b_i}$ are chosen form $A$, and player II wins if the assignment $s(\vec{x_i})=\vec{a_i}$, $s(\vec{y_i})=\vec{b_i}$ satisfies for all $n$: $$\mA\models_s\psi(\vec{x_n},\vec{y_n})\wedge\bigwedge_{=\!(\vec{w}_0^p,y_{0,p})\in S}(\vec{w}_i ^p=\vec{w}_{n} ^p\to y_{i,p}=y_{n,p}).  $$ We can  get a winning strategy for player II as follows. Since $\mA\models\phi$, there are functions $f_i(x_1,...,x_m)$, $1\le i\le n$, such that if $X$ is the set of all assignments $s$ with $s(y_i)=f_i(s(x_1),...,s(x_m))$, then
\begin{equation}\label{fg}\mA\models_X\bigwedge_{1\le j\le k}\dep(\vec{w}^{i_j},y_{0,i_j})\wedge \theta(\vec{x},\vec{y}).
\end{equation}
The strategy of player II in $G(\mA,\Phi)$ is to play $$\vec{b_i}=(f_1(\vec{a_i}),\ldots,f_n(\vec{a_i})).$$ This guarantees that if $s(\vec{x_i})=\vec{a_i}$, $s(\vec{y_i})=\vec{b_i}$, then  clearly $\mA\models_s\psi(\vec{x_n},\vec{y_n})$, but we have to also show that $\mA\models_s\bigwedge_{=\!(\vec{w}_0^p,y_{0,p})\in S}(\vec{w}_i ^p=\vec{w}_{n} ^p\to y_{i,p}=y_{n,p})$. This follows immediately from (\ref{fg}).

Suppose then $\mA$ is a countable model of $\Phi$. We show $\mA\models\phi$. Let $\vec{a_n}=(a_n^1,...,a_n^m)$, $n<\omega$, list all $m$-sequences of elements of $A$. We play the game $G(\mA,\Phi)$ letting player I play the sequence $\vec{a_n}$ as his $n$'th move. Let $\vec{b_n}$ be the response of $II$, according to her winning strategy, to $\vec{a_n}$. Let $f_i$ be the function $\vec{a_n}\mapsto b^i_n$. It is a direct consequence of the winning condition of player II that if $X$ is the set of assignments $s$ with $s(y_i)=f_i(s(x_1),...,s(x_m))$, $1\le i\le n$, then (\ref{fg}) holds.
 \end{proof}

%

\begin{definition} A model $\mA$ is \emph{recursively saturated} if it satisfies
\[\forall\vec{x}((\bigwedge_n\exists y\bigwedge_{m\le n}\phi_m(\vec{x},y))\to
\exists y\bigwedge_n\phi_n(\vec{x},y)),\]
 whenever $\{\phi_n(\vec{x},y):n\in\oN\}$ is recursive.
\end{definition}

There are many recursively saturated models:
\begin{proposition}[\cite{MR0403952}]
For every infinite $\mA$ there is a recursively saturated countable $\mA'$ such that $\mA\equiv \mA'$.
\end{proposition}

Over a  recursively saturated model, the game expression can be replaced
by a conjunction of its approximations:

\begin{proposition}\label{recsat}
If $\mA$  is  recursively saturated (or finite), then
\[\mA\models\Phi\leftrightarrow\bigwedge_n\Phi^n.\]
\end{proposition}

\begin{proof}
Suppose first $\mA\models\Phi$. Thus $\exists$ has a winning strategy $\tau$ in the game $G(\mA,\Phi)$. Then $\mA\models\Phi^n$ for each $n$, since the $\exists$-player can simply follow the strategy $\tau$ in the game $G(\mA,\Phi^n)$ and win. For the converse, suppose $\mA\models\Phi^n$ for each $n$.
Let
$\Phi_m^{n+1}(\vec{x_0},\vec{y_0},...,\vec{x}_{n-1},\vec{y}_{n-1})$ be the first-order formula:

\[
\begin{array}{l}
\psi(\vec{x_0},\vec{y_0})\wedge\\
\psi(\vec{x_1},\vec{y_1})\wedge\bigwedge_{=\!(\vec{w}_0 ^p,y_{0,p})\in S}(\vec{w}_0 ^p=\vec{w}_1 ^p\to y_{0,p}=y_{1,p})\wedge\\
\psi(\vec{x_2},\vec{y_2})\wedge
\bigwedge_{i\in\{0,1\}}\bigwedge_{=\!(\vec{w}_0^p,y_{0,p})\in S}
(\vec{w}_i ^p=\vec{w}_2 ^p\to y_{i,p}=y_{2,p})\wedge\\
\cdots\\
\cdots\\
\forall\vec{x}_{n-1}\exists\vec{y}_{n-1}(\\
\psi(\vec{x}_{n-1},\vec{y}_{n-1})\wedge  \bigwedge_{i=0}^{n-1} \bigwedge_{=\!(\vec{w}_0^p,y_{0,p})\in S}(\vec{w}_i ^p=\vec{w}_{n-1} ^p\to y_{i,p}=y_{n-1,p})    ))\wedge\\
\cdots\\
\cdots\\
\forall\vec{x}_{m}\exists\vec{y}_{m}(\\
\psi(\vec{x}_{m},\vec{y}_{m})\wedge  \bigwedge_{i=0}^{m-1} \bigwedge_{=\!(\vec{w}_0^p,y_{0,p})\in S}(\vec{w}_i ^p=\vec{w}_{m} ^p\to y_{i,p}=y_{m,p})    )\ldots)\\
\end{array}
\]

%
%
%

The strategy of $\exists$ in the game $G(\mA,\Phi)$ is the following:
\begin{description}

\item[$(*)$]If the game position is $(\vec{a_0},\vec{b_0},...,\vec{a}_{n-1},\vec{b}_{n-1})$ then for each $m$ we have $$\mA\models\Phi_m^n(\vec{a_0},\vec{b_0},...,\vec{a}_{n-1},\vec{b}_{n-1}).$$ 

\end{description}
It is easy to  check, using the recursive saturation,  that $\exists$ can play according to this strategy and that she wins this way.\end{proof}

\begin{corollary}
If $\mA$ is a countable recursively saturated (or finite) model, then \[\mA\models\phi\leftrightarrow\bigwedge_n\Phi^0_n.\]
\end{corollary}

\begin{proof} By Propositions \ref{17} and \ref{recsat}.
\end{proof}














\begin{lemma}[Transitivity of deduction]\label{ded1}Suppose $\phi_1,...,\phi_n$,$\psi_1,...,\psi_m$ and $\theta$ are sentences of dependence logic.
If $\{\phi_1,...,\phi_n\}\vdashd\psi_i$ for $i=1,...,m$, and $\{\psi_1,...,\psi_m\}\vdashd\theta$, then $\{\phi_1,...,\phi_n\}\vdashd\theta$.
\end{lemma}

\begin{proof}
The deduction of $\theta$ from $\{\phi_1,...,\phi_n\}$ is obtained from the dedction of $\theta$ from $\{\psi_1,...,\psi_m\}$ by replacing each application of the assumption $\psi_i$ by the deduction of $\psi_i$ from $\{\phi_1,...,\phi_n\}$. \qed
\end{proof}

We are now ready to prove the main result of this article.
\begin{theorem} Let $T$ be a set of sentences of dependence logic and
$\phi\in \FO$. Then the following are equivalent:
\begin{description}
\item[(I)] $T\models\phi$
\item[(II)] $T\vdashd\phi$
\end{description}
\end{theorem}

\begin{proof}
Suppose first (I) but $T\not\vdashd\phi$. Let $T^*$ consist of all the approximations of the dependence sentences in $T$. Since the approximations are provable from the original sentences, Lemma~\ref{ded1} gives $T^*\not\vdashd\phi$.  Note that $T^*\cup\{\neg\phi\}$ is a first order theory. Clearly,   $T^*\cup\{\neg\phi\}$ is deductively consistent in first order logic, since we have all the first order inference rules as part of our deduction system. Let $\mA$ be a countable recursively saturated model of this theory. By Lemma~\ref{recsat}, $\mA\models T\cup\{\neg\phi\}$, contradicting (I). We have proved (II).  (II) implies (I) by the Soundness Theorem.
\end{proof}

\section{Examples and open questions}


In this section we present  some  examples   and open problems.

\begin{example} This example is an application of the dependence distribution rule in a context where continuity of functions is being discussed.
\begin{enumerate}
\item Given $\epsilon,x,y$ and $f$.
\item If $\epsilon>0$, then  there is $\delta>0$ depending only on $\epsilon$ such that  if $|x-y|<\delta$, then $|f(x)-f(y)|<\epsilon$.
\item Therefore,  there is $\delta>0$ depending only on $\epsilon$ such that if $\epsilon>0$ and $|x-y|<\delta$, then $|f(x)-f(y)|<\epsilon$.
\end{enumerate}
\end{example}

\begin{example} This is an example of an application of the  dependence elimination rule, again in a context where continuity of functions is being contemplated.
\begin{enumerate}

\item Assume that for every $x$ and every $\epsilon>0$ there is $\delta>0$ depending only on $\epsilon$ such that for all $y$, if $|x-y|<\delta$, then $|f(x)-f(y)|<\epsilon$.

\item Therefore, for every $x$ and every $\epsilon>0$ there is $\delta>0$ such that for all $y$, if $|x-y|<\delta$, then $|f(x)-f(y)|<\epsilon$, and moreover, for another $x'$ and $\epsilon'>0$  there is $\delta'>0$ such that for all $y'$, if $|x'-y'|<\delta'$, then $|f(x')-f(y')|<\epsilon$, and if $\epsilon=\epsilon'$, then $\delta=\delta'$.
\end{enumerate}
\end{example}

\begin{example} This is a different type of example of the use of the dependence elimination rule. Let $\phi$ be the following sentence:
\[ \phi: \forall x \exists y\exists z(\dep(y,z)\wedge (x=z \wedge y\neq c) ), \]
where $c$ is a constant symbol. It is straightforward to verify that $\mA\models \phi$ iff $A$ is infinite. The idea is that
the dependence atom $\dep(y,z)$ forces the interpretation of $y$ to encode an injective function from $A$ to $A$ that is not surjective (since $y\neq c$ must also hold). On the other hand, for the approximations $\Phi^n$, it holds that  $\mA\models \Phi^n$ iff $|A|\ge n+1$:  for  $\Phi^1$  
\[ \Phi^1:= \forall x_0 \exists y_0\exists z_0(x_0=z_0 \wedge y_0\neq c) \]
this is immediate, and in general, the claim can be proved using induction on $n$.
\end{example}

We end this section with some  open questions.

\begin{enumerate}
\item Our complete axiomatization, as it is, applies only to sentences. Is the same axiomatization complete also with respect to formulas?
\item Our natural deduction makes perfect sense also as a way to derive non first order sentences of dependence logic. What is the modified concept (or concepts) of a structure relative to which this is complete?
\item Are dependence distribution and dependence delineation  really necessary? Do we lose completeness if one or both of them are dropped? Is there other redundancy in the rules?
\item Do similar axiom systems yield Completeness Theorems in other dependence logics, such as modal dependence logic, or dependence logic with intuitionistic implication? 
\item Is there a similar deductive system for first order consequences of sentences of independence logic introduced in \cite{GV}? In principle this should be possible.  One immediate complication that arises is that independence logic does not satisfy downward closure (the analogue of Proposition \ref{Downward closure}), and hence, e.g., the rule  $\forall $ E 
is not sound for independence logic. For example, the formula  
\(\forall x \forall y (x\bot y),  \)
is universally true but $x\bot y$ certainly is not. 

\end{enumerate}

\vspace{1cm} \noindent
\begin{tabular}{l}
{\it Juha Kontinen}\\
Department of Mathematics and Statistics\\
University of Helsinki, Finland\\
\tt{juha.kontinen@helsinki.fi}\\
\\
{\it Jouko V\"a\"an\"anen}\\
Department of Mathematics and Statistics\\
University of Helsinki, Finland\\
and\\
Insitute for Logic, Language
 and Computation\\
University of Amsterdam,
 The Netherlands\\
\tt{jouko.vaananen@helsinki.fi}\\
\end{tabular}


\begin{thebibliography}{1}

\bibitem{MR0465788}
Jon Barwise.
\newblock Some applications of {H}enkin quantifiers.
\newblock {\em Israel J. Math.}, 25(1-2):47--63, 1976.

\bibitem{MR0403952}
Jon Barwise and John Schlipf.
\newblock An introduction to recursively saturated and resplendent models.
\newblock {\em J. Symbolic Logic}, 41(2):531--536, 1976.

\bibitem{ADJK}
Arnaud Durand and Juha Kontinen.
\newblock Hierarchies in dependence logic.
\newblock {\em arXiv:1105.3324v1}.

\bibitem{btxdoc}
Pietro Galliani.
\newblock Entailment semantics for independence logic.
\newblock Manuscript, 2011.

\bibitem{GV}
Erich Gr\"adel and Jouko V\"a\"an\"anen.
\newblock Dependence and independence.
\newblock {\em Studia Logica}, to appear.

\bibitem{MR1410063}
Jaakko Hintikka.
\newblock {\em The principles of mathematics revisited}.
\newblock Cambridge University Press, Cambridge, 1996.
\newblock With an appendix by Gabriel Sandu.

\bibitem{MR0209138}
Lars Svenonius.
\newblock On the denumerable models of theories with extra predicates.
\newblock In {\em Theory of {M}odels ({P}roc. 1963 {I}nternat. {S}ympos.
  {B}erkeley)}, pages 376--389. North-Holland, Amsterdam, 1965.

\bibitem{MR2351449}
Jouko V{\"a}{\"a}n{\"a}nen.
\newblock {\em Dependence logic}, volume~70 of {\em London Mathematical Society
  Student Texts}.
\newblock Cambridge University Press, Cambridge, 2007.

\bibitem{MR0409106}
Robert Vaught.
\newblock Descriptive set theory in {$L_{\omega_1\omega }$}.
\newblock In {\em Cambridge {S}ummer {S}chool in {M}athematical {L}ogic
  ({C}ambridge, {E}ngland, 1971)}, pages 574--598. Lecture Notes in Math., Vol.
  337. Springer, Berlin, 1973.

\end{thebibliography}
\end{document}